\documentclass{article}
\usepackage{amssymb, amsrefs, amsmath, amsthm, amsfonts, graphicx, amscd}

\usepackage[margin=1in]{geometry}
\numberwithin{equation}{section}

\newtheorem{theorem}{Theorem}[section]
\newtheorem{lemma}[theorem]{Lemma}
\newtheorem{proposition}[theorem]{Proposition}

\newtheorem{predefinition}[theorem]{Definition}
\newenvironment{definition}{\begin{predefinition}\rm}{\end{predefinition}}
\newtheorem{preremark}[theorem]{Remark}
\newenvironment{remark}{\begin{preremark}\rm}{\end{preremark}}
\newtheorem{prenotation}[theorem]{Notation}
\newenvironment{notation}{\begin{prenotation}\rm}{\end{prenotation}}
\newtheorem{preexample}[theorem]{Example}

\newtheorem{preclaim}[theorem]{Claim}

\newtheorem{prequestion}[theorem]{Question}

\begin{document}


\author{Robert Guralnick, Beth Malmskog, and Rachel Pries}
\title{The Automorphism Groups of a Family of Maximal Curves}

\maketitle

\begin{abstract}
The Hasse Weil bound restricts the number of points of a curve which are defined over a 
finite field; if the number of points meets this bound, the curve is called maximal.  
Giulietti and Korchmaros introduced a curve $\mathcal{C}_3$ which is maximal over ${\mathbb F}_{q^6}$
and determined its automorphism group.
Garcia, Guneri, and Stichtenoth generalized this construction to a family of curves $\mathcal{C}_n$, 
indexed by an odd integer $n \geq 3$, 
such that ${\mathcal C}_n$ is maximal over ${\mathbb F}_{q^{2n}}$.
In this paper, we determine the automorphism group $\textrm{Aut}(\mathcal{C}_n)$ 
when $n > 3$; in contrast with the case $n=3$, it fixes the point at infinity on $\mathcal{C}_n$. 
The proof requires a new structural result about automorphism groups of curves in characteristic $p$ 
such that each Sylow $p$-subgroup has exactly one fixed point.\\
keywords: Weil bound, maximal curve, automorphism, ramification.\\
MSC:11G20, 14H37.
\end{abstract}

\section{Introduction}\label{intro}

In this paper, we prove a group-theoretic result which
produces a new structural result about automorphism groups of curves $\mathcal{X}$ in positive characteristic $p$. 
These theorems are closely related to the main results of \cite{asch, GKnontame, KOS}. 
The main idea is to analyze the automorphism group $A$ of a curve $\mathcal{X}$ having genus at least $2$ under the following conditions:
there is a Sylow $p$-subgroup $Q$ of $A$, which contains an elementary abelian subgroup of order $p^2$, such that the 
action of $Q$ fixes exactly $1$ point of $\mathcal{X}$ and whose other orbits on $\mathcal{X}$ all have size $|Q|$.
The result is that the center $M$ of $N_A(Q)$ is a (possibly trivial) prime-to-$p$ subgroup which is normal in $A$, 
and either $A$ fixes a point of $\mathcal{X}$ or $A/M$ is almost simple and there is a short list of possibilities
for its socle.  The proof uses TI subgroups and the classification of finite simple groups.

Our motivation for this result was to determine the automorphism groups of the infinitely many 
new maximal curves discovered by \cite{GK09} and \cite{GGS08}.
Chinburg has already found another application of this result
in classifying automorphism groups of Katz-Gabber covers. 

To describe our main application, let $q=p^h$ be a power of a prime and 
let $\mathcal{X}$ be a smooth connected projective curve of genus $g$ defined over $\mathbb{F}_{q^2}$. 
The Hasse-Weil bound states that the number of points of $\mathcal{X}$ defined over $\mathbb{F}_{q^2}$ is bounded above by $q^2+1+2gq$.  
A curve which attains this bound is called an $\mathbb{F}_{q^2}$-maximal curve.   
The curve $\mathcal{H}_q$ with affine equation $x^q+x-y^{q+1}=0$ is known as the Hermitian curve and has been well studied.  
It is maximal over $\mathbb{F}_{q^2}$, and thus maximal over $\mathbb{F}_{q^{2n}}$ for $n\geq 3$ odd \cite[VI, 4.4]{Stich93}.  
It has genus $q(q-1)/2$, the highest genus which is attainable for an $\mathbb{F}_{q^2}$-maximal curve.  
It is the unique ${\mathbb F}_{q^2}$-maximal curve of this genus.
  
Let $n \geq 3$ be odd and let $m=(q^n+1)/(q+1)$. 
The curve $\mathcal{X}_n$ with affine equation $y^{q^2}-y-z^{m}=0$ has genus $(q-1)(q^n-q)/2$.
In {\cite[Thm.\ 1]{ABQ}}, the authors proved that ${\mathcal X}_n$ is $\mathbb{F}_{q^{2n}}$-maximal.  

For $n \geq 3$ odd, define $\mathcal{C}_n$ to be the normalization of the fiber product of the covers of curves 
$\mathcal{H}_q\rightarrow\mathbb{P}^1_y$ and $\mathcal{X}_n\rightarrow \mathbb{P}^1_y$.  
The curve $\mathcal{C}_3$ is isomorphic to the curve introduced in \cite{GK09}.  
Giulietti and Korchmaros proved that $\mathcal{C}_3$ is maximal over ${\mathbb F}_{q^6}$
using the natural embedding theorem \cite{KTem} 
which states that every $\mathbb{F}_{q^2}$-maximal curve is isomorphic to a smooth absolutely irreducible curve of degree $q+1$ 
embedded in a non-degenerate Hermitian variety.  
They also proved that $\mathcal{C}_3$ is not covered by any Hermitian curve.  
Garcia, Guneri, and Stichtenoth proved that the genus of $\mathcal{C}_n$ is $(q-1)(q^{n+1}+q^n-q^2) / 2$ 
and that $\mathcal{C}_n$ is $\mathbb{F}_{q^{2n}}$-maximal for $n\geq 3$ \cite{GGS08}. 
Recently, Duursma and Mak proved that $\mathcal{C}_n$ is not Galois covered by the Hermitian curve 
$\mathcal{H}_{q^{2n}}$ if $q$ is odd \cite{Duursma}. 

In this paper, we study the geometric automorphism groups of the curves $\mathcal{C}_n$
which we denote by $\textrm{Aut}({\mathcal C}_n)$.
In \cite[Thm.\ 6]{GK09}, the authors determined the automorphism group $\textrm{Aut}(\mathcal{C}_3)$ 
when $q \equiv 1 \bmod 3$ and found a normal subgroup of index 3 in $\textrm{Aut}(\mathcal{C}_3)$ if $q\equiv 2 \bmod 3$.
This automorphism group is very large compared to the genus $g_{\mathcal{C}_3}$ of $\mathcal{C}_3$, 
i.e., $\textrm{Aut}(\mathcal{C}_3)\geq 24g_{\mathcal{C}_3}(g_{\mathcal{C}_3}-1)$.
There is no point of $\mathcal{C}_3$ fixed by the full automorphism group. 

In this paper, we determine the automorphism group of $\mathcal{C}_n$ for $n > 3$ odd.
A consequence of our results is that all automorphisms of ${\mathcal C}_n$ are defined over ${\mathbb F}_{q^{2n}}$.
The structure of the automorphism group when $n > 3$ turns out to be quite different from the case when $n=3$. 
 
\begin{theorem} \label{Tmaintheorem}
Suppose $n > 3$ is odd.
The automorphism group $\textrm{Aut}(\mathcal{C}_n)$ fixes the point at infinity on $\mathcal{C}_n$ 
and is a semi-direct product $\Gamma$ of the form $Q\rtimes \Sigma$ where 
$\Sigma$ is a cyclic group of order $(q^n+1)(q-1)$ and $Q$ is a non-abelian group of order $q^3$.
\end{theorem}

We describe the structure of $\Gamma$ precisely in Section \ref{Q}.
Since $|\textrm{Aut}(\mathcal{C}_n)|=q^3(q^n+1)(q-1)$ and the genus of $\mathcal{C}_n$ is $(q-1)(q^{n+1}+q^n-q^2)/2$,
Theorem \ref{Tmaintheorem} implies that  $|\textrm{Aut}(\mathcal{C}_n)|$ grows more than linearly in $g(\mathcal{C}_n)$, 
surpassing the Hurwitz bound for large $q$.  
Since $|Q|$ is relatively small, Theorem \ref{Tmaintheorem} also shows that the order of the 
automorphism group is mostly prime-to-$p$.
This contrasts with the situation in many recent papers about curves whose automorphism group is large 
\cite{lehrmat, rocher1, rocher2}.    

Here is the outline of the proof.
Using \cite[Thm.\ 12.11]{HKT08}, we first show that
$\textrm{Aut}(\mathcal{X}_n)$ fixes the point at infinity on $\mathcal{X}_n$ 
and is a semi-direct product of the form $ (\mathbb{Z}/p)^{2h}\rtimes \Sigma$
where $\Sigma$ is a cyclic group of order $(q^n+1)(q-1)$, see Proposition \ref{autxn}.
The next step is to study the inertia group $I_{\mathcal{C}_n}$ at the point at infinity $P_{\infty}$ 
for the cover $\mathcal{C}_n \to \mathcal{C}_n/\textrm{Aut}(C_n)$.
Using ramification theory, we prove in Proposition \ref{TautC} that $I_{\mathcal{C}_n}$ has the structure of the 
semi-direct product $\Gamma = Q \rtimes \Sigma$ described in Theorem \ref{Tmaintheorem}.
By studying the orbits of $\Gamma$ on ${\mathcal C}_n$, we prove that the Sylow $p$-subgroups of $\textrm{Aut}(C_n)$ 
are isomorphic to $Q$ in Proposition \ref{Csyl}.
The first difference from the case $n=3$ occurs in Proposition \ref{Pliftw}, where we prove that 
not all automorphisms of the Hermitian curve $H_q$ lift to automorphisms of ${\mathcal C}_n$ when $n >3$.

To finish the proof of Theorem \ref{Tmaintheorem}, we apply the new structural result, Theorem \ref{Tbob},
about automorphism groups of curves in positive characteristic as described above.
This structural result relies on the group-theoretic result found in Theorem \ref{TI}.
The case $q=2$ is handled separately in Proposition \ref{Pq=2}.
See Sections \ref{Sbob} and \ref{Sq=2} for precise details.
All these results are combined in Section \ref{Smaintheorem} to prove Theorem \ref{Tmaintheorem}.

This paper can be found on the archive at arXiv:1105.3952; 
in November 2011, G\"uneri, \"Ozdemir, and Stichtenoth sent us an alternative proof of Theorem \ref{Tmaintheorem} which uses Riemann-Roch spaces \cite{GOS}.

We would like to thank Tim Penttila for helpful conversations. 
The first author was partially supported by NSF grant DMS-1001962.
The third author was partially supported by NSF grant DMS-1101712.

\section{Automorphisms of $\mathcal{C}_n$ and $\mathcal{X}_n$}

\subsection{Geometry of $\mathcal{C}_n$ and $\mathcal{X}_n$}

Recall that $q=p^h$, that $n \geq 3$ is odd, and that $m=(q^n+1)/(q+1)$. 
Recall that $\mathcal{H}_q$ and $\mathcal{X}_n$ are the smooth projective curves with affine equations 
\begin{eqnarray}
\mathcal{H}_q &:& x^q+x-y^{q+1}=0,\\
\mathcal{X}_n &:& y^{q^2}-y-z^{m}=0.
\end{eqnarray}
Define $\mathcal{C}_n$ to be the normalization of the fiber product of the covers of curves 
$\mathcal{H}_q\rightarrow\mathbb{P}^1_y$ and $\mathcal{X}_n\rightarrow \mathbb{P}^1_y$
as illustrated in the following diagram:

\[\begin{array}{cccc}
\mathcal{C}_n=& \mathcal{X}_n \tilde{\times}_{\mathbb{P}^1}\mathcal{H}_q & \rightarrow& \mathcal{X}_n\\
& \downarrow& & \downarrow\\
& \mathcal{H}_q & \longrightarrow& \mathbb{P}^1_y.
\end{array}
\]

\begin{remark} 
The curve $\mathcal{C}_3$ was initially presented as the intersection in $\mathbb{P}^3$ of two hypersurfaces \cite{GK09}.
The projective curve in $\mathbb{P}^3$ given by the homogenization of the equations $x^q+x-y^{q+1}$ and $y^{q^2}-y-z^m$ 
is smooth only when $n=3$.  For $n\geq5$, the curve has a cusp type singularity.  
\end{remark}

Let $\infty_y$ be the point at infinity on ${\mathbb P}^1_y$.

\begin{lemma}\label{cninfinity}
The curve $\mathcal{X}_n$ has a unique point $\infty_{{\mathcal X}_n}$ in the fiber above $\infty_y$.
The curve $\mathcal{C}_n$ has a unique point $P_\infty$ in the fiber above $\infty_y$.
\end{lemma}

\begin{proof}
The ${\mathbb Z}/m$-Galois cover $\mathcal{X}_n \to {\mathbb P}^1_y$ is totally ramified above $\infty_y$.
The cover $\mathcal{H}_q\rightarrow\mathbb{P}^1_y$ is totally ramified above $\infty_y$ 
with ramification index $q$.
Since $q$ and $m$ are relatively prime, $qm$ divides the ramification index of any point over $\infty_y$ in the cover $\mathcal{C}_n\rightarrow\mathbb{P}^1_y$.  Since the degree of $\mathcal{C}_n\rightarrow\mathbb{P}^1_y$ is $qm$, 
this implies that there can only be a single point $P_\infty$ in $\mathcal{C}_n$ above $\infty_y$. 
\end{proof}

Applying the Riemann-Hurwitz formula to the cover $\mathcal{C}_n \to \mathcal{H}_q$ 
shows that the genus of $\mathcal{C}_n$ is $(q-1)(q^{n+1}+q^n-q^2) / 2$.  In \cite{GGS08},  
Garcia, Guneri, and Stichtenoth used the fiber product construction to study the fibers above $\mathbb{F}_{q^{2n}}$-points of $\mathcal{H}_q$ 
and $\mathcal{X}_n$, and thereby proved that $\mathcal{C}_n$ is $\mathbb{F}_{q^{2n}}$-maximal. 

To summarize, the genera and numbers of $\mathbb{F}_{q^{2n}}$-points for the curves $\mathcal{X}_n$ and $\mathcal{C}_n$ are:
\begin{center}
$\begin{array}{ll}
g_{\mathcal{X}_n}=(q-1)(q^n-q)/2, & \#\mathcal{X}_n(\mathbb{F}_{q^{2n}})=q^{2n+1}-q^{n+2}+q^{n+1}+1,\\
g_{\mathcal{C}_n}=(q-1)(q^{n+1}+q^n-q^2) / 2, & \#\mathcal{C}_n(\mathbb{F}_{q^{2n}})=q^{2n+2}-q^{n+3}+q^{n+2}+1.
\end{array}$
\end{center}

\subsection{Subgroups of ${\textrm{Aut}}(\mathcal{C}_n)$}\label{Q}

Let $a,b\in\mathbb{F}_{q^2}$ be such that $a^q +a = b^{q+1}$.  Define 

\[Q_{a,b}=\left(\begin{array}{ccc}1 & b^q & a \\0 & 1 & b \\0 & 0 & 1\end{array}\right).\]

Let $Q=\{Q_{a,b} \mid a,b\in\mathbb{F}_{q^2},\  a^q +a = b^{q+1} \}$.  
Note that $Q$ is a subgroup of the special unitary group SU$(3,q^2)$ and that 
\[ Q_{a,b}Q_{c,d}= \left(\begin{array}{ccc}1 & (b+d)^q & a+c+b^qd \\0 & 1 & b+d \\0 & 0 & 1\end{array}\right)=Q_{a+c+b^qd, b+d}.
\]
This implies that $Q$ is not abelian, since $Q_{a,b}Q_{c,d}=Q_{c,d}Q_{a,b}$ if and only if $b^qd=d^qb$, which is not true for arbitrary $b,d\in\mathbb{F}_{q^2}$.

The order of $Q$ is $q^3$ since there is a bijection between $Q$ and the $\mathbb{F}_{q^2}$-rational affine points of $\mathcal{H}_q$.  It is known that $Q$ has exponent $p$ if $p\neq 2$ and exponent $4$ if $p=2$.  The center of $Q$ is 
$Z=\{Q_{a,0} \mid a \in \mathbb{F}_{q^2},\  a^q+a=0\}$.

\begin{lemma} \label{Lcenter}
The subgroup $Z$ is isomorphic to $(\mathbb{Z}/p)^h$.
\end{lemma}

\begin{proof}
First, $|Z|=q=p^h$ since $a^q+a=\text{Tr}(a)=0$ has $q$ solutions $a\in\mathbb{F}_{q^2}$.
To check that $Z$ is abelian, note that
\[Q_{a,0}Q_{b,0}=Q_{a+b,0}=Q_{b,0}Q_{a,0}.\]
Finally, $Z$ has exponent $p$ because $Q$ has exponent $p$ when $p \not = 2$
and because $Q_{a,0}^2=Q_{0,0}$ when $p=2$.
\end{proof}

\begin{lemma}
The quotient group $Q/Z$ is isomorphic to $(\mathbb{Z}/p)^{2h}$.
\end{lemma}

\begin{proof}
First, $Q/Z$ is abelian since for $Q_{a,b}, Q_{c,d}\in Q$, the commutator $Q_{a,b}^{-1}Q_{b,c}^{-1}Q_{a,b}Q_{c,d}$ is in $Z$:  
\begin{eqnarray*}
Q_{a,b}^{-1}Q_{b,c}^{-1}Q_{a,b}Q_{c,d}&=&Q_{a^q+c^q+b^qd ,-b-d}Q_{a+c+b^qd, b+d}\\
&=&Q_{a^q+a+c^q+c+2b^qd-(b+d)^{q+1},0}\in Z.
\end{eqnarray*}
The order of $Q/Z$ is $p^{2h}$ by Lemma \ref{Lcenter}.
Finally, $Q/Z$ has exponent $p$ because $Q$ has exponent $p$ when $p \not = 2$ and because
$Q_{a,b}^2=Q_{b^{q+1}, 0} \in Z$ when $p=2$.
\end{proof}

The group $Q$ acts on the curve $\mathcal{C}_n$.  Let $(x,y,z)$ denote an affine point of $\mathcal{C}_n$.  Define \[Q_{a,b}: \hspace{.25in} x\mapsto x+b^qy+a,  \hspace{.25in} y\mapsto y+b, \hspace{.25in}  z \mapsto z.\]  

\begin{lemma}\label{Qinaut}
The group $Q$ is contained in $\textrm{Aut}(\mathcal{C}_n)$ and the quotient curve $\mathcal{C}_n/Q$ is a projective line.  
Furthermore, $\mathcal{C}_n/Z=\mathcal{X}_n$.
\end{lemma}

\begin{proof}
Note that $Q$ stabilizes $\mathcal{H}_q$.  Let $(x,y)$ be an affine point of $\mathcal{H}_q$.  Then
\begin{eqnarray*}
Q_{a,b}(x)^q+Q_{a,b}(x)-Q_{a,b}(y)^{q+1}&=&x^q+x+b^{q^2}y^q+b^q y +a^q+a- (y+b)^{q+1}\\
 &=& y^{q+1}+by^q+b^q y+b^{q+1}- (y+b)^{q+1}\\
 &=& 0.
\end{eqnarray*}
In addition, $Q$ stabilizes $\mathcal{X}_n$.   Let $(y,z)$ be an affine point of $\mathcal{X}_n$.  Then
\begin{eqnarray*}
 Q_{a,b}(y)^{q^2}- Q_{a,b}(y)- Q_{a,b}(z)^{m}&=&y^{q^2}-y+b^{q^2}-b-z^{m}\\
  &=&y^{q^2}-y-z^{m}\\
  &=&0.
  \end{eqnarray*}
So $Q$ stabilizes the fiber product $\mathcal{C}_n$.  

The quotient curve $\mathcal{C}_n/Z$ is $\mathcal{X}_n$ because $\mathbb{K}(\mathcal{X}_n)$ is fixed by $Z$ and 
$|\mathbb{K}(\mathcal{C}_n):\mathbb{K}(\mathcal{X}_n)|=q=|Z|$.  
A similar argument shows that $\mathbb{K}(\mathcal{C}_n/Q)\cong\mathbb{K}(z)$, and so $\mathcal{C}_n/Q$ is a projective line, 
denoted $\mathbb{P}^1_z$.
\end{proof}

Let $\zeta\in \mu_{(q^n+1)(q-1)}$ be a $(q^n+1)(q-1)$-st root of unity.  Define $g_{\zeta}$ by
\[g_{\zeta}: \hspace{.25in}  x \mapsto  \zeta^{q^n+1} x, \hspace{.5in} y \mapsto \zeta^{m}y, \hspace{.5in} z \mapsto \zeta z.\] 

Note that $\zeta^{q^n+1}$ is a $(q-1)$-st root of unity, so an element of $\mathbb{F}_q$.  Therefore $(\zeta^{q^n+1})^q=\zeta^{q^n+1}$.  
The group $\Sigma=\{g_{\zeta} \mid \zeta\in\mu_{(q^n+1)(q-1)}\}$ is a cyclic group of order $(q^n + 1)(q-1)$.  
Define $M \subset \Sigma$ to be the subgroup of order $m=(q^n+1)/(q+1)$ and $N \subset \Sigma$ to be the subgroup of order $q^n+1$.  

\begin{lemma} \label{Ginaut}
The group $\Sigma$ is contained in $\textrm{Aut}(\mathcal{C}_n)$.  
The quotient curves $\mathcal{C}_n/\Sigma$ and $\mathcal{C}_n/N$ are projective lines and $\mathcal{C}_n/M=\mathcal{H}_q$.
\end{lemma}

\begin{proof}
Note that $\Sigma$ stabilizes $\mathcal{H}_q$.   Let $(x,y)$ be an affine point of $\mathcal{H}_q$.  Then
\begin{eqnarray*}
g_{\zeta}(x)^q+g_{\zeta}(x)-g_{\zeta}(y)^{q+1}&=&\zeta^{q(q^n+1)}x^q+\zeta^{q^n+1}x-\zeta^{q^n+1} y^{q+1}\\
&=&\zeta^{q^n+1}(x^q+x-y^{q+1})\\
 &=& 0.
\end{eqnarray*}
Also $\Sigma$ stabilizes $\mathcal{X}_n$.  Let $(y,z)$ be an affine point of $\mathcal{X}_n$. Then
\begin{eqnarray*}
g_{\zeta}(y)^{q^2}-g_{\zeta}(y)-g_{\zeta}(z)^m&=&\zeta^{mq^2} y^{q^2}-\zeta^m y-\zeta^m z^m\\
&=&\zeta^m (\zeta^{m(q^2-1)}y^{q^2}-y-z^{m})\\
&=&\zeta^m( y^{q^2}-y-z^{m})\\
&=& 0.
\end{eqnarray*}

So $\Sigma$ stabilizes the fiber product $\mathcal{C}_n$.  Therefore $\Sigma\subset \textrm{Aut}(\mathcal{C}_n)$.

As before, $\mathbb{K}(\mathcal{C}_n/M)\cong\mathbb{K}(\mathcal{H}_q)$ and $\mathbb{K}(\mathcal{C}_n/N)\cong\mathbb{K}(x)$.  
So $\mathcal{C}_n/N$ is a projective line denoted $\mathbb{P}^1_x$.   
Also $\mathbb{K}(\mathcal{C}_n/\Sigma)\cong \mathbb{K}(u)$, where $u=x^{q-1}$, 
and so $\mathcal{C}_n/\Sigma$ is a projective line denoted $\mathbb{P}_u^1$.  
\end{proof}

\begin{lemma}\label{semidirect}
There is a homomorphism
$\phi:\Sigma\rightarrow \textrm{Aut}(Q)$ given by $g_{\zeta}\mapsto \psi_{\zeta}$, 
where $\psi_{\zeta}:Q_{a,b}\mapsto g_{\zeta}Q_{a,b}g_{\zeta}^{-1}$. 
\end{lemma}

\begin{proof}
Let $g_{\zeta}\in \Sigma$ and $Q_{a,b}\in Q$ as above.  Then
\begin{center}
$\begin{array}{cccc}

\psi_{\zeta}(Q_{a,b}): & x &\mapsto & x+b^q\zeta^{q^n+1-m}y+a\zeta^{q^n+1}\\

& y & \mapsto & y+\zeta^{m}b\\

& z & \mapsto & z.
\end{array}$
\end{center}
Since $\zeta^{q^n+1-m}=(\zeta^m)^q$, this means that $\psi_{\zeta}(Q_{a,b})=Q_{\zeta^{q^n+1}a, \zeta^{m} b}$.  
This is a well-defined element of $Q$ because 
\[(\zeta^{m} b)^{q+1}=\zeta^{q^n+1}b^{q+1}=\zeta^{q^n+1}(a+a^q)=\zeta^{q^n+1}a+(\zeta^{q^n+1}a)^q.\]  
\end{proof}

By Lemma \ref{semidirect}, the group generated by $Q$ and $\Sigma$ is a semi-direct product 
$\Gamma:= Q \rtimes_\phi \Sigma$.  

\begin{proposition} \label{Pautsub}
The group $\textrm{Aut}(\mathcal{C}_n)$ contains a subgroup isomorphic to the semi-direct product 
$\Gamma= Q \rtimes_{\phi} \Sigma$.
\end{proposition}

\begin{proof}
By Lemmas \ref{Qinaut} and \ref{Ginaut}, $Q$ and $\Sigma$ are contained in $\text{Aut}(\mathcal{C}_n)$.
By Lemma \ref{semidirect}, $\Sigma$ normalizes $Q$ in $\text{Aut}(\mathcal{C}_n)$.
Since $|Q|$ and $|\Sigma|$ are relatively prime, 
the group generated by $\Sigma$ and $Q$ in $\text{Aut}(\mathcal{C}_n)$ is a semi-direct product 
$\Gamma=Q\rtimes_{\phi} \Sigma$. 
\end{proof}

\subsection{Quotients of ${\mathcal C}_n$}

The semi-direct product $\Gamma=Q\rtimes_{\phi} \Sigma$ is not a direct product, 
but it does contain several subgroups which are direct products.

\begin{lemma}
An element $g \in \Sigma$ commutes with every element of $Q$ (resp.\ $Z$) if and only if 
$g \in M$ (resp.\ $N$).
The group $\Gamma$ contains subgroups isomorphic to $Q\times M$ and $Z\times N$.  
The center of $\Gamma$ is $M$.
\end{lemma}

\begin{proof}
Notice that 
\[g_{\zeta}Q_{a,b}g_{\zeta}^{-1}=Q_{\zeta^{q^n+1}a, \zeta^{m} b}.\]
For $b \neq 0$, we have $g_{\zeta}Q_{a,b}g_{\zeta}^{-1}=Q_{a,b}$ if and only if $g_{\zeta} \in M$.  
If $b=0$, then $g_{\zeta}Q_{a,0}g_{\zeta}^{-1}=Q_{a,0}$ if and only if $g_{\zeta}\in N$.   
All the claims follow from these conjugation computations.
\end{proof}

The following diagram summarizes the coverings described above:

\begin{center}
$\begin{array}{ccccccc}
& & q & & q^2 & & \\
&\mathcal{C}_n&\rightarrow&\mathcal{X}_n  & \rightarrow&\mathbb{P}^1_z &\\

m&\downarrow& & \downarrow& &\downarrow&\\
&\mathcal{H}_q&\rightarrow&\mathbb{P}^1_y&\rightarrow&\mathbb{P}^1_t &\\
q+1 &\downarrow& &\downarrow && &\\
&\mathbb{P}^1_x & \rightarrow&\mathbb{P}^1_w&  & &\\
q-1 & \downarrow& && & &\\
& \mathbb{P}^1_u&&&
\end{array}$
\end{center}
The numbers next to the arrows are the degrees of the coverings.  
The projective line $\mathbb{P}^1_w$ denotes the curve $\mathcal{C}_n/(Z\times N)$, 
where $\mathbb{K}(\mathcal{C}_n/(Z\times N))\cong \mathbb{K}(w)$ with $w=y^{q+1}$.  
The projective line $\mathbb{P}^1_t$ denotes the curve $\mathcal{C}_n/(Q\times M)$, 
where $\mathbb{K}(\mathcal{C}_n/(Q\times M))\cong \mathbb{K}(t)$ with $t=z^{m}$.  

\subsection{The automorphism group of $\mathcal{X}_n$} \label{SautX}

\begin{proposition} \label{autxn}
The automorphism group of $\mathcal{X}_n$ fixes the point at infinity and 
$\textrm{Aut}(\mathcal{X}_n)=(Q/Z) \rtimes_{\phi} \Sigma$.
\end{proposition}

\begin{proof}
Let $I_{\mathcal{X}_n}$ be the inertia group at the point at infinity $\infty_{\mathcal{X}_n}$ 
of the cover $\mathcal{X}_n \to \mathcal{X}_n/\text{Aut}(\mathcal{X}_n)$.  
The curve $\mathcal{X}_n$ is defined by the equation $A(y)=B(z)$, where $A(y)=y^{q^2}-y$ and $B(z)=z^{m}$.
Notice that $A(y)$ has the property that $A(y+a)=A(y)+A(a)$ for $a \in {\mathbb F}_{q^2}$.  
Automorphisms of this kind of curve are studied in \cite[Chapter 12]{HKT08} 
(under the unnecessary hypothesis that they are defined over ${\mathbb F}_{q^2}$.)
By \cite[Thm.\ 12.11]{HKT08}, all (geometric) automorphisms of ${\mathcal X}_n$ fix $\infty_{{\mathcal X}_n}$, 
so $\text{Aut}(\mathcal{X}_n)=I_{\mathcal{X}_n}$.

We follow the proof of \cite[Thm.\ 12.7]{HKT08} to describe automorphisms $\alpha \in I_{\mathcal{X}_n}$.
First, $I_{\mathcal{X}_n}$ is a semi-direct product of a cyclic subgroup $H$ of order prime-to-$p$ 
with a normal subgroup $G_1$ which is a $p$-group by \cite[Chapter IV, Cor.\ 4]{SerreLF}.
The automorphism $\alpha$ preserves each of the linear series $|q^2\infty_{\mathcal{X}_n}|$ and $|m \infty_{\mathcal{X}_n}|$.
If $m > q^2$ (the other case is similar), this implies that $\alpha(z)= cz +d$ and $\alpha(y)=ay+Q(z)$ with 
$a,c,d \in \overline{\mathbb F}_p$ and $Q(z) \in  \overline{\mathbb F}_p[z]$ with ${\rm deg}(Q(z))< m/q^2$.
Then $A(ay+Q(z))-B(cz+d)$ must be a constant multiple of $A(y)-B(z)$.  
Studying the leading terms shows that $a^{q^2}=a=c^m$ and so $a \in {\mathbb F}_{q^2}$ and $c \in {\mathbb F}_{q^{2n}}$.
The next term shows that $d=0$ and $Q(z)$ is a constant in ${\mathbb F}_{q^2}$.
Thus $|H|$ divides $(q-1)(q^n+1)$ and $|G_1|$ divides $q^2$.
Now $Z$ is normal in $\Gamma$ and $\mathcal{X}_n=\mathcal{C}_n/Z$, 
so $(Q \rtimes_{\phi} \Sigma)/Z \subset \text{Aut}(\mathcal{X}_n)$.
Thus $I_{\mathcal{X}_n} =(Q/Z) \rtimes_{\phi} \Sigma$.
\end{proof}

\section{The automorphism group of ${\mathcal C}_n$} \label{ramification}

Giulietti and Korchmaros determined the automorphism group for the curve $\mathcal{C}_3$. 

\begin{theorem}  \label{TGK} \cite[Thm.\ 6]{GK09}
 If $q\equiv 1 \bmod 3$, then 
 \[\textrm{Aut}(\mathcal{C}_3)\cong \textrm{SU}(3,q^2)\times \mathbb{Z}/((q^3+1)/(q+1)).\]  
 
 If $q\equiv 2 \bmod 3$, then there exists $G\vartriangleleft \textrm{Aut}(\mathcal{C}_3$) 
such that  $|\textrm{Aut}(\mathcal{C}_3):G|=3$ and
 \[G\cong \textrm{SU}(3,q^2)\times \mathbb{Z}/((q^3+1)/3(q+1)).\]
\end{theorem}

In this section, we determine the automorphism group of $\mathcal{C}_n$ for $n \geq 5$ odd.
The strategy is first to use ramification theory to show that $\Gamma$ is the inertia group at $P_\infty$, see Proposition \ref{TautC}.
Second, by studying the orbits of $\Gamma$ on $\mathcal{C}_n$, 
we prove that the Sylow $p$-subgroups of $\textrm{Aut}(C_n)$ 
are isomorphic to $Q$ in Proposition \ref{Csyl}.
We produce new structural results about automorphism groups of curves in positive characteristic in Section \ref{Sbob},  
which show, in particular, that $M$ is normal in $\textrm{Aut}(C_n)$.
Determining the automorphism group thus reduces to determining which automorphisms of the Hermitian curve $H_q$ 
lift to automorphisms of ${\mathcal C}_n$, which we answer in Proposition \ref{Pliftw}.

\subsection{Background on higher ramification groups}

A good reference about ramification theory is \cite[Chapter IV]{SerreLF}. 

\begin{definition}
Let $f:\mathcal{Y}^{\prime}\rightarrow\mathcal{Y}$ be a $G$-Galois cover of curves.  
Let $P^{\prime}$ be a point on $\mathcal{Y}^{\prime}$ and let $P=f(P^{\prime})$.
For $i\geq -1$, define the $i$-th ramification group at $P^{\prime}$ by 
\[G_i(P^{\prime}|P)=\{ \sigma \in G \mid v_{P^{\prime}}(\sigma(z)-z)\geq i+1 \hspace{.1in} \forall  \hspace{.1in}  z\in\mathcal{O}_{P^{\prime}}\}.\]
\end{definition}

An integer $i$ for which $G_i\neq G_{i+1}$ is a {\it lower jump} of the filtration of higher ramification groups at $P^{\prime}$.
Here are some facts about these ramification groups.

\begin{lemma} \label{Lramfact}
\begin{description}
\item{(i)} $G\geq G_{-1}\trianglerighteq G_0 \trianglerighteq ... \trianglerighteq G_i \trianglerighteq G_{i+1} \trianglerighteq ...$ and 
$G_N=\{\textrm{id}\}$ for $N >>0$.

\item{(ii)} $|G_0|=e(P^{\prime}|P)$ is the ramification index of $P^{\prime}$ over $P$.  

\item{(iii)} Let $H\leq G$ and let $H_i$ denote the $i$-th ramification group for the $H$-Galois cover 
${\mathcal Y}^{\prime} \to {\mathcal Y}^{\prime}/H$.  Then $H_i=G_i\cap H$.  

\item{(iv)} The order of $G_1$ is a power of $p$, the quotient $G_0/G_1$ is cyclic with prime-to-$p$ order, and $G_i/G_{i+1}$ is elementary abelian of exponent $p$ for $i\geq 1$.

\item{(v)} If $s\in G_i, t\in G_j$, and $i,j\geq 1$, then $sts^{-1}t^{-1}\in G_{i+j+1}$.
\end{description}
\end{lemma}

\begin{proof}
These results can be found in \cite[Chapter IV, Sections 1 \& 2]{SerreLF}.
\end{proof}

The ramification groups can also be indexed by another system, known as the upper numbering.  
To define the upper numbering, for $u \in {\mathbb R}^{\geq -1}$, 
let $G_u=G_i$, where $i=\lceil u \rceil$.  Define piecewise linear functions $\phi$ and $\psi$ by
\[
\phi(u)=\int^u_0\frac{dt}{(G_0:G_t)}
\]
and $\psi(u)=\phi^{-1}(u)$.  Define the ramification groups in the upper numbering by $G^v=G_{\psi(v)}$.
An index $j$ for which $G^j\neq G^k$ for all $k > j$ is an {\it upper jump} of the filtration of higher ramification groups.  
When the ramification index is a power of $p$, an immediate consequence of the definition 
is that the smallest jump is the same in the upper and lower numbering.
Upper numbering behaves well with respect to quotient groups of $G$.

\begin{lemma} \label{Lramup}
If $H$ is a normal subgroup of $G$ and $(G/H)^v$ denotes the $v$-th ramification group for the 
$G/H$-Galois cover ${\mathcal Y}^{\prime}/H \to {\mathcal Y}$ then $(G/H)^v = (G^v H)/H$.
\end{lemma}

\begin{proof}
This can be found in \cite[Chapter IV, Prop.\ 14]{SerreLF}.
\end{proof}

\subsection{Filtrations at infinity}

\begin{lemma}\label{jumpcntoxn}
There is one break in the ramification filtration of $\mathcal{C}_n\rightarrow\mathcal{X}_n$ at $P_\infty$
and it occurs at index $q^n+1$ in the lower numbering.
\end{lemma}

\begin{proof}
The degree $q$ cover $\mathcal{C}_n\rightarrow\mathcal{X}_n$ has affine equation $x^q+x=y^{q+1}$.
Let $v_{\mathcal{X}_n}$ denote the valuation at the point of infinity $\infty_{{\mathcal X}_n}$ on $\mathcal{X}_n$.
Then $v_{\mathcal{X}_n}(y)=-m$, since $m$ is the ramification index of ${\mathcal X}_n \to {\mathbb P}^1_y$ above $\infty_y$.
Thus $v_{\mathcal{X}_n}(y^{q+1}) = -(q^n+1)$.
Since $q^n+1$ is prime-to-$p$, and since $x^q+x$ is a separable additive polynomial with all its roots in ${\mathbb F}_{q^2}$, 
the result follows from \cite[Prop.\ 3.7.10]{Stich93}.
\end{proof}

\begin{lemma}\label{ext2}
There is one break in the ramification filtration of $\mathcal{X}_n\rightarrow\mathbb{P}^1_z$ 
at $\infty_{{\mathcal X}_n}$ and it occurs at index $m$ in the lower numbering.
\end{lemma}

\begin{proof}
The degree $q^2$ cover $\mathcal{X}_n\rightarrow {\mathbb P}^1_z$ has affine equation $y^{q^2}-y=z^m$.
The valuation of $z^m$ at the point of infinity on ${\mathbb P}^1_z$ is $-m$.
Since $m$ is prime-to-$p$ and $y^{q^2}-y$ is a separable additive polynomial with all its roots in ${\mathbb F}_{q^2}$, 
the result again follows from \cite[Prop.\ 3.7.10]{Stich93}.
\end{proof}

\begin{remark}
Here is an alternative way to prove Lemmas \ref{jumpcntoxn} and \ref{ext2}.
One can show that there is one break in the ramification filtration of $\mathcal{H}_q\rightarrow\mathbb{P}^1_y$ 
(resp.\ $\mathbb{P}^1_y \to \mathbb{P}^1_t$) 
and it occurs at index $q+1$ (resp.\ 1) in the lower numbering.
Pulling these covers back by a cover which has tame ramification at $\infty_y$ multiples the lower jump by the 
ramification index, which is $m$. 
\end{remark}

\begin{proposition}\label{jumpcntop1z}
There are two jumps in the ramification filtration of $\mathcal{C}_n\rightarrow\mathbb{P}^1_z$ at $P_\infty$
 and they occur at indices $m=(q^n+1)/(q+1)$ and $q^n+1$ in the lower numbering.
\end{proposition}

\begin{proof}
The upper jump of the cover $\mathcal{X}_n \to \mathbb{P}^1_z$ is $\phi(m)= \sum_{i=1}^{m}1 =m$.
Thus $m$ is an upper jump for $\mathcal{C}_n \to \mathbb{P}_z^1$ by Lemma \ref{Lramup}.
By Lemma \ref{Lramfact}(iii), the lower jump $q^n+1$ for $\mathcal{C}_n \to \mathcal{X}_n$ 
is also a lower jump for $\mathcal{C}_n \to \mathbb{P}_z^1$. 

Since $m \not = q^n+1$, there are exactly two jumps in the ramification filtration of $\mathcal{C}_n \to \mathbb{P}_z^1$
by Lemmas \ref{Lramfact}(iii) and \ref{Lramup}.
Because the ramification index is a power of $p$, the smaller jump is the same in the upper and lower numbering, 
and thus equals $m$, completing the proof.
\end{proof}

\subsection{The inertia group of $\mathcal{C}_n$} \label{Siner}

Let $I_{\mathcal{C}_n}$ be the inertia group at $P_{\infty}$ for the cover $\mathcal{C}_n \to \mathcal{C}_n/\text{Aut}(\mathcal{C}_n)$.
Let $S$ be the Sylow-$p$ subgroup of $I_{\mathcal{C}_n}$. 

\begin{proposition}\label{centerS}
The group $Z$ is in the center of $S$.
\end{proposition}

\begin{proof}
Let $\mathcal{W}=\mathcal{C}_n/S$.  
Consider the cover $\mathcal{C}_n \to \mathcal{W}$ which is totally ramified at $P_\infty$.  
Let $s$ be an element of $S-Q$ with maximal lower jump $J$ in this cover.   
By Lemma \ref{Lramfact}(iv), $s$ and $s^p$ have different lower jumps.  
The jump for $s^p$ must be greater than that for $s$, so $s^p\in Q$.  
That means that $|s|=p$, or $|s|=p^2$, or $p=2$ and $|s|=8$.  
In all three cases, it suffices to show that $J \leq q^n+1$,
because this shows that $Z$ is contained in the last non-trivial ramification group 
and thus $Z$ is in the center of $S$ by Lemma \ref{Lramfact}(v).

{\bf Case 1:} $|s|=p$.  Assume that $J> q^n+1$.  
Then, by hypothesis, $s$ is contained in the last non-trivial ramification group, 
and thus commutes with $Z$ by Lemma \ref{Lramfact}(v).
Therefore $s$ descends to an automorphism $\overline{s}$ in the inertia group $I_{\mathcal{X}_n}$ at $\infty_{\mathcal{X}_n}$.
This gives a contradiction since the Sylow $p$-subgroup of $I_{\mathcal{X}_n}$ is $Q/Z$ by Proposition \ref{autxn} 
and $s \notin Q$.  Thus $J\leq q^n+1$. 

{\bf Case 2:} $|s|=p^2$. Then $\langle s \rangle \cap Q = \langle s^{p} \rangle$, where $ \langle s^{p} \rangle\cong \mathbb{Z}/p$.  
That means that in the ramification filtration of $\mathcal{C}_n \to \mathbb{P}^1_z$, 
the lower jump of $s^p$ is $q^n+1$ or $m$.  
Since $Q \subset S$, Lemma \ref{Lramfact}(iii) implies that the lower jumps of elements of $Q$ will the same in the cover 
$\mathcal{C}_n \to \mathcal{W}$.  Since the lower jump of $s$ is less than that of $s^{p}$, 
this implies $J < q^n+1$.

{\bf Case 3}: $p=2$ and $|s|=8$.
Then $\langle s \rangle \cap Q = \langle s^{2} \rangle$, where $\langle s^{2} \rangle\cong \mathbb{Z}/4$.
Then $s^2 \not \in Z$ so the lower jump of $s^2$ is $m$.
This implies that $J < m$. 
\end{proof}

\begin{proposition} \label{TautC}
The inertia group $I_{\mathcal{C}_n}$ of $\mathcal{C}_n$ at $P_\infty$ is $\Gamma=Q\rtimes_{\phi} \Sigma$. 
\end{proposition}

\begin{proof} 
Suppose $s \in S$.  
By Proposition \ref{centerS}, $s$ commutes with $Z$ and so $s$ descends to an automorphism 
$\overline{s}$ of ${\mathcal X}_n$.  By Proposition \ref{autxn}, $\overline{s} \in Q/Z$ and so 
$S/Z=Q/Z$.  The Third Isomorphism Theorem implies that $S=Q$.  

Let $T$ be a subgroup of $I_{\mathcal{C}_n}$ isomorphic to the prime-to-$p$ group $I_{\mathcal{C}_n}/Q$ and containing $M$.  
Since $T$ is cyclic, the automorphisms in $T$ descend to $\mathcal{H}_q$.  
The prime-to-$p$ part of $I_{\mathcal{H}_q}$ is isomorphic to $\Sigma/M$ by \cite[Eqn.\ 2.3]{GSX}, so $T/M=\Sigma/M$ 
and the Third Isomorphism Theorem implies that $T=\Sigma$.   
Thus $I_{\mathcal{C}_n}=Q\rtimes_{\phi} \Sigma$.
\end{proof}

\subsection{The Sylow $p$-subgroup of $\text{Aut}(\mathcal{C}_n)$} \label{Ssylow}

More information can be gained by considering the orbits of ${\mathcal C}_n({\mathbb F}_{q^{2n}})$ 
under $\Gamma$.  

\begin{proposition}\label{orbits}
\begin{description}
\item{(i)}
Suppose $\eta_1=(x_1,y_1,z_1)$ and $\eta_2=(x_2,y_2,z_2)$ are two affine $\mathbb{F}_{q^{2n}}$-points of ${\mathcal C}_n$.
Then $\eta_1$ and $\eta_2$ are in the same orbit under $\Gamma$ if and only if 
$z_2=\zeta z_1$ for some $(q^n+1)(q-1)$-th root of unity $\zeta$. 
\item{(ii)} The orbits of ${\mathcal C}_n(\mathbb{F}_{q^{2n}})$ under $\Gamma$ consist of:
one orbit $O_{\infty}=\{P_{\infty}\}$ of cardinality $1$; one orbit of cardinality $q^3$ which is
$$O_{0} = {\mathcal C}_n(\mathbb{F}_{q^2})-\{P_{\infty}\}
=\{(x,y,0) \mid x,y \in {\mathbb F}_{q^2}, \ x^q+x=y^{q+1}\};$$
and $(q^{n-1}-1)/(q-1)$ orbits of cardinality $|\Gamma|$.  
\item{(iii)} The cover ${\mathcal C}_n \to {\mathcal C}_n/\Gamma$ is a cover of the projective line 
${\mathbb P}^1_s$ where $s=z^{(q^n+1)(q-1)}$.  It is ramified only above $s=0$ and $s=\infty$.
Above $s=0$, the ramification is tame of order $(q^n+1)(q-1)$.
Above $s=\infty$, it is totally ramified and the lower jumps in the 
ramification filtration are $m$ and $q^n+1$.
\end{description}
\end{proposition}

\begin{proof}
Consider the action of $\Gamma$ on the  
set $\Omega$ of affine $\mathbb{F}_{q^{2n}}$-points of ${\mathcal C}_n$.
To begin, consider the action of $Q \subset \Gamma$ on $\Omega$.  
Suppose $\eta_1 = (x_1,y_1,z_1) \in \Omega$.
The orbit $O'_{\eta_1}$ of $\eta_1$ under $Q$ is the intersection of $\Omega$ with the fiber of the $Q$-Galois cover 
${\mathcal C}_n \to {\mathbb P}^1_z$ above $z_1$.
The orbit $O'_{\eta_1}$ has cardinality $q^3$ because ${\mathcal C}_n \to {\mathbb P}^1_z$ is unramified away from $\infty_z$.   
Two points $\eta_1, \eta_2 \in \Omega$ are in the same orbit under $Q$ if and only if $z_1=z_2$.

Each orbit of $\Gamma$ on $\Omega$ is a union of some of the orbits of $Q$ on $\Omega$.
The problem of studying the orbit of $\eta_1$ under $\Gamma$ 
thus reduces to studying the action of $G$ on the variable $z$.
Note that $z_2=g(z_1)$ for some $g \in G$ if and only if $z_2=\zeta z_1$ 
for some $(q^n+1)(q-1)$th root of unity $\zeta$. 
Thus $\eta_1, \eta_2 \in \Omega$ are in the same orbit of $\Gamma$ on $\Omega$ if and only if 
$z_2=\zeta z_1$ 
for some $(q^n+1)(q-1)$-th root of unity $\zeta$. 
This proves part (i).

For part (ii), by definition, the action of $\Gamma$ fixes $P_{\infty}$.
The affine points in ${\mathcal C}_n({\mathbb F}_{q^2})$ are exactly the points
$\eta=(x,y,0)$ with $x, y \in {\mathbb F}_{q^2}$ and $x^q+x=y^{q+1}$.
By part (i), this yields an orbit of cardinality $q^3$.

Finally, if $\eta=(x,y,z) \in \Omega$ with $z \not = 0$, then the orbit of $\eta$ consists of 
$|\Gamma|$ points.  The number of orbits of $\Gamma$ on $\Omega-{\mathcal C}_n({\mathbb F}_{q^2})$
is thus 
$$(\#\Omega-q^3)/|\Gamma|=(q^{2n+2}-q^{n+3}+q^{n+2}-q^3)/|\Gamma|=(q^{n-1}-1)/(q-1).$$ 

Part (iii) is immediate from part (ii) and Proposition \ref{jumpcntop1z}.
\end{proof}

\begin{proposition} \label{Csyl}
The Sylow-$p$ subgroups of $\textrm{Aut}(\mathcal{C}_n)$ are isomorphic to $Q$.
\end{proposition}

\begin{proof}
By Proposition \ref{orbits}(ii), $Q$ has exactly one fixed point $P_\infty$ on ${\mathcal C}_n$ and all other orbits have size $q^3$.  
The normalizer $N_A(Q)$ fixes the unique fixed point of $Q$ and so is contained in the inertia group of $P_\infty$.
If $Q$ were properly contained in a Sylow $p$-subgroup $P$,
then $N_P(Q)$ would properly contain $Q$.
However, $p$ does not divide $[N_A(Q):Q]$ by Proposition \ref{TautC}.
This would give a contradiction and so $Q$ is a Sylow $p$-subgroup of ${\text Aut}(C_n)$.
\end{proof}

\subsection{Lifting automorphisms of the Hermitian curve}

In this section, we show that there are automorphisms of $\mathcal{H}_q$ 
which do not lift to automorphisms of $\mathcal{C}_n$ when $n > 3$. 
The strategy is to consider a specific involution $\omega \in {\rm Aut}(\mathcal{H}_q)$.
If $\omega$ lifts to an automorphism $\tilde{\omega}$ of $\mathcal{C}_n$, then 
$\tilde{\omega}$ exchanges $P_\infty$ with another point;
using valuation theory, we show this implies $n=3$.

Let $P_0 \in {\mathcal C}_n$ be the point $(x,y,z)=(0,0,0)$. 
Let $v_0$ (resp.\ $v_\infty$) denote the valuation of $\mathcal{C}_n$ at $P_0$ (resp.\ $P_\infty$).
Let $t=z^{q^{n-3}}/x$.

\begin{lemma} \label{Lvaluation}
\begin{description}
\item{(i)} At the point $P_0$, the functions $x,y,z$ have the following valuations:
$$v_0(y)=m, \ v_0(x)=q^n+1, \ v_0(z)=1.$$
\item{(ii)} At the point $P_\infty$, the functions $x,y,z,t$ have the following valuations:
$$v_\infty(y)=-qm, \ v_\infty(x)=-(q^n+1), \ v_\infty(z)=-q^3, \ v_\infty(t)=1.$$
\end{description}
\end{lemma}

\begin{proof}
\begin{description}
\item{(i)} The functions $x,y,z$ all equal $0$ at $P_0$, so their valuations at $P_0$ are all positive.  
The ramification index of $P_0$ above the point $y=0$ in ${\mathbb P}^1_y$ equals $m$, 
and so $v_0(y)=m$.  The relation $x^q+x=y^{q+1}$ shows that $v_0(x)=m(q+1)=q^n+1$.  
The relation $z^m=y^{q^2}-y$ shows that $v_0(z)=1$.
\item{(ii)} The functions $x,y,z$ all have poles at $P_\infty$ so their valuations at $P_\infty$ are all negative.
The ramification index of $P_\infty$ above the point $\infty_y$ in ${\mathbb P}^1_y$ equals $qm$, 
and so $v_\infty(y)=-qm$.  
The relation $x^q+x=y^{q+1}$ shows that $v_\infty(x)=q^n+1$.  
The relation $z^m=y^{q^2}-y$ shows that $v_\infty(z)=-q^3$.
The valuation of $t$ at $P_\infty$ is $-q^3q^{n-3}+(q^n+1)$ and so $t$ is a uniformizer at $P_\infty$.
\end{description}
\end{proof}

Consider the automorphism $\omega$ of $\mathcal{H}_q$ given by $\omega(x)=1/x$ and $\omega(y)=y/x$.

\begin{proposition} \label{Pliftw}
If $n >3$, then the automorphism $\omega$ of $\mathcal{H}_q$ does not lift to an automorphism of $\mathcal{C}_n$.
\end{proposition}

\begin{proof}
Suppose that $\omega$ lifts to an automorphism $\tilde{\omega}$ of $\mathcal{C}_n$.  
Since $\omega(x)=1/x$, it follows that $\tilde{\omega}(P_\infty)=P_0$.

Applying $\omega$ to the equation $z^m=y^{q^2}-y$ implies that $\tilde{\omega}(z)^m=(\frac{y}{x})^{q^2}-\frac{y}{x}$.
By Lemma \ref{Lvaluation}, $v_\infty(\tilde{\omega}(z))=1$.
Thus $\tilde{\omega}(z)=t f$ for some function $f \in {\mathcal O}(C)$ which has no pole or zero at $P_\infty$.
Letting $P=(x,y,z)$, this means that $v_\infty(f(P))=0$.

Now $\omega^2={\rm id} \in {\rm Aut}(\mathcal{H}_q)$ so $\tilde{\omega}^2 \in M$.
Since $\tilde{\omega}^2 \in M$, this implies that $\tilde{\omega}^2(z)=\zeta z$ for some $m$th root of unity $\zeta$.
Thus $v_\infty(\tilde{\omega}^2(z))=-q^3$ and $v_0(\tilde{\omega}^2(z))=1$ by Lemma \ref{Lvaluation}.

\paragraph{Claim: The function $f$ has a zero of order $(q^{n-3}-1)(q^3+1)$ at $P_0$.}

To prove the claim, a computation shows that:
\begin{equation} \label{Eformula}
\tilde{\omega}^2(z)=\tilde{\omega}(t f)=x^{-(q^{n-3}-1)} z^{q^{2n-6}} f(P)^{q^{n-3}} f(\omega(P)).
\end{equation}
Taking the valuation of both sides of Equation \ref{Eformula} at $P_\infty$ yields that:
$$-q^3=v_\infty(\tilde{\omega}^2(z))=-q^n+q^{n-3}-1+v_\infty(f(\omega(P))).$$
Now there is an isomorphism of local rings 
$\tilde{\omega}^*: {\mathcal O}_{\mathcal{C}_n, \omega(P)} \to  {\mathcal O}_{\mathcal{C}_n, P}$.
In particular, this implies that $v_0(f(P))=v_\infty(f(\omega(P))=(q^{n-3}-1)(q^3+1)$, which completes the proof of the claim.

Now taking the valuation of both sides of Equation \ref{Eformula} at $P_0$ yields that:
$$1=v_0(\tilde{\omega}^2(z))=q^{2n-6}-(q^n+1)(q^{n-3}-1)+q^{n-3}v_0(f(P))+v_0(f(\tilde{\omega}(P))).$$
Now $v_0(f(\tilde{\omega}(P))=v_\infty(f(P))=0$.
Substituting for $v_0(f(P))$ and simplifying yields that $0=2q^{n-3}(q^{n-3}-1)$, which implies that $n=3$.
\end{proof}

\subsection{A structural result on automorphism groups of curves} \label{Sbob}

The purpose of this section and the next is to prove that the subgroup $M$ is normal in the full automorphism group 
of ${\mathcal C}_n$.
We start with some group theoretic preliminaries, refering the reader
to \cite{aschbook} for definitions and terminology.  Fix a prime $p$.
A {\it $p'$-group} is a group of order prime-to-$p$.   If $\pi$ is a set of primes, then
$O_{\pi}(J)$ is the unique maximal normal subgroup of $J$ that is a $\pi$-group
(i.e. has order divisible only by primes in $\pi$).   A group $J$ is {\it almost simple}
if it has a unique minimal normal subgroup $S$ with $S$ nonabelian simple;
we say that $S$ is the {\it socle} of $J$.
Thus,  $S \subset J \subset \textrm{Aut}(S)$.
For completeness, we include a proof of the following well-known result, e.g., \cite[Thm.\ 6.21]{Isaacs}.

\begin{lemma} \label{Lbob}  Let $Q$ be a $p$-group which is not cyclic or
generalized quaternion (for $p=2$).
If $Q$ acts on a $p'$-group $R$, then $R =\langle C_R(h) | 1 \ne h \in Q \rangle$.
\end{lemma}

\begin{proof}   We can replace $Q$ by any subgroup still satisfying
the hypotheses and so assume that $Q$ is elementary abelian of order $p^2$.  First consider
the case that $R$ is an $r$-group for some prime
$r$ (necessarily $r \ne p$).    Let $\Phi(R)$ denote the Frattini
subgroup of $R$
and set ${\bar R}=R/\Phi(R)$.   Since $C_{\bar R}(Q_0) =
C_R(Q_0)\Phi(R)/\Phi(R)$ for any subgroup $Q_0$ of $Q$, it
suffices to assume that $\Phi(R)=1$, i.e., $R$ is an elementary abelian
$r$-group.
By Maschke's theorem, it suffices to assume that $R$ is irreducible.
By Schur's Lemma,
$Q$ cannot act faithfully on $R$ and so $R=C_R(h)$ for some nontrivial
$h \in Q$.

Now consider the general case.  Let $r$ be a prime divisor of $|R|$.
By the Sylow Theorem,
$Q$ normalizes some Sylow $r$-subgroup $L$ of $R$.  Applying the first
paragraph shows
that $L = \langle C_L(h) | 1 \ne h \in Q \rangle$.   Thus, $\langle
C_R(h) | 1 \ne h \in Q \rangle$
contains a Sylow $r$-subgroup of $R$ for each prime divisor $r$ of
$|R|$, whence the result.
\end{proof}

We now identify certain groups. 
Recall that a subgroup $H$ of a finite group $G$ is called a TI subgroup
if for $y \in G \setminus{N_G(H)}$,  $H \cap H^y=1$.   

\begin{theorem} \label{TI}  Let  $Q$ be a Sylow $p$-subgroup of a finite
group $A$ and suppose that $Q$ is not cyclic or generalized quaternion (if $p=2$).  
Assume that $I=N_A(Q) = QC$ with $C$ cyclic.
Assume also that $Q$ is a TI subgroup of $A$ and $Q \ne A$.   
Set $M=Z(I)$.  Then $M = O_{p'}(A)$, and $A/M$ is almost simple
and the socle of $A/M$ is isomorphic to one of:
\begin{enumerate}
\item   ${\rm PSL}_2(p^a), a \ge 2$;
\item   ${\rm PSU}_3(p^a), p^a > 2$;
\item   ${\rm Sz}(2^{2a +1}),  p=2, a > 1$; or
\item   ${^2}G_2(3^{2a+1})', p = 3$.
\end{enumerate}
In particular,  $A$ acts $2$-transitively on the set of Sylow $p$-subgroups of $A$
\end{theorem}

\begin{remark}
Theorem \ref{TI} partially extends the main result of \cite{KOS} 
(where the permutation action of $A$ on the set of Sylow $p$-subgroups 
was assumed to be doubly transitive) and it follows from \cite{KOS} under that permutation condition.
See \cite{asch}  which generalized the result of \cite{KOS} in a different direction.  
Note that the proof of Theorem \ref{TI} uses the classification of finite simple groups
whereas the results in \cite{asch} and \cite{KOS} do not. 
\end{remark}

\begin{proof}   Since $Q$ is a TI-subgroup, it follows that $O_p(A)=1$.
Indeed, this implies that $C_A(h) \subset I$ for any $1 \ne h \in Q$ (since
$Q$ is the unique Sylow $p$-subgroup containing $h$).

Now apply Lemma \ref{Lbob} to conclude that $R:=O_{p'}(A) \subset O_{p'}(I)$.  
Since $R$ normalizes $Q$ and $Q$ normalizes $R$, it follows that 
$[Q,R] \subset Q \cap R=1$, whence $R \subset M$.

Set $J=A/R$.  Then $O_{p'}(J)=1$.   We identify $Q$ with its image in 
$J$.  If $Q_1$ is any nontrivial subgroup of $Q$, then $N_J(Q_1) \subset N_A(Q_1)R/R 
\subset R$.  Thus, we can apply \cite[7.6.1]{GLS} to conclude
that  $J$ is almost simple and given as in that theorem.
Now use the fact that $I= \langle N_J(Q_1), 1 \ne  Q_1 \subset Q \rangle$
(and  $I/Q$ is cyclic) to conclude from
 \cite[7.6.2]{GLS} that only 7.6.1(a) can hold.
 
Finally, note that $Z(I/R)=1$ in each case, whence $R=M$. 
This completes the proof.
\end{proof} 

We apply this in the situation of a group acting on a curve with the following setup:
  
\begin{notation} \label{Nsetup}
Let $k$ be a field of characteristic $p$.
Let $\mathcal{X}$ be a curve of genus at least $2$ and set $A=\textrm{Aut}_k(\mathcal{X})$.
Let $Q$ be a $p$-subgroup of $A$ and let $I=N_A(Q)$.  We assume the following conditions:

\begin{description}
\item{(i)}  $Q$ is a Sylow $p$-subgroup of $A$;
\item{(ii)}  $Q$ is not cyclic or generalized quaternion (if $p=2$) --
equivalently,  $Q$ contains an elementary abelian subgroup of order $p^2$.
\item{(iii)}   $Q$ fixes precisely $1$ point $x \in \mathcal{X}$ and acts semiregularly
on $\mathcal{X} - \{x\}$ (i.e. every other orbit has size $|Q|$);
\end{description}
\end{notation}

In the situation of Notation \ref{Nsetup}, since $Q$ has a unique fixed point $x$, 
then $I$ also fixes $x$ and thus $I$ is the inertia group of $x$.
Note that condition (iii) is equivalent to the condition that every element of order $p$
has a unique fixed point.
This condition implies that, 
for any $H \subset {\rm Aut}(\mathcal{X})$ containing a Sylow $p$-subgroup,
the cover $\mathcal{X} \to \mathcal{X}/H$ has one wildly ramified branch point, with the inertia group containing a Sylow $p$-subgroup,
and all other branch points are tamely ramified.

We obtain:

\begin{theorem} \label{Tbob}
With notation as in \ref{Nsetup}, suppose that $A \ne I$.  Then the center $M$ of $I$ is a 
(possibly trivial) $p'$-subgroup normal in $A$ and $A/M$ is almost simple with socle one of:
\begin{enumerate}
\item   ${\textrm PSL}_2(p^a), a \ge 2$;
\item   ${\textrm PSU}_3(p^a), p^a > 2$;
\item   ${\textrm Sz}(2^{2a +1}),  p=2, a > 1$; or
\item   ${^2}G_2(3^{2a+1})', p = 3$.
\end{enumerate}
\end{theorem}

\begin{proof}
Note that if $a \in A \setminus{I}$, then $Q^a$ fixes $a(x) \ne x$, whence
$Q \cap Q^a = 1$.  Thus, $Q$ is a TI subgroup of $A$.  The result then follows from Theorem \ref{TI}.
\end{proof}

\subsection{The case $q=2$} \label{Sq=2}

When $q=2$, condition (ii) of Notation \ref{Nsetup} is not satisfied, and so different methods are needed.  
Recall that $\Gamma$ is the inertia group of ${\mathcal C}_n$ at $P_\infty$ and that $M$ is the center of $\Gamma$.
Let $A=\mathrm{Aut}(\mathcal{C}_n)$ denote the automorphism group of $\mathcal{C}_n$.

\begin{proposition} \label{Pq=2}
Suppose $q=2$ and $n>3$ is odd.  Then $M$ is normal in $A = \mathrm{Aut}(\mathcal{C}_n)$
and either $A=\Gamma$ or $[A:\Gamma]=3$.
\end{proposition}

\begin{proof}
By Proposition \ref{Csyl}, a Sylow $2$-subgroup $Q$ of $A$ is a quaternion group of order $8$.
Proposition \ref{orbits} and Proposition \ref{TautC} imply that 
every nontrivial element of $Q$ has a unique fixed point $x$ on the curve.

Let $z$ be the central involution in $Q$.  By Glauberman's $Z^*$ Theorem
\cite{GG}, $z$ is central in $A/O(A)$ where $O(A) = O_{2'}(A)$.
By the Sylow Theorem, this implies that $A = O(A)C_A(z)$.  Since $z$ has a unique
fixed point,  $C_A(z)=I$.    In particular, this implies that $A$ is solvable.
Thus $A$ contains an odd order subgroup $H$ of index $8$ in $A$.

By the Hurwitz bound, $|H|  \le 84(g_{\mathcal{C}_n}-1)$.  
When $q=2$, then $g_{\mathcal{C}_n}=2^n+2^{n-1}-2$ and $|\Gamma \cap H|=2^n+1$.
This implies that $[H:\Gamma \cap H] < 126$.  

Suppose that a prime $r > 3$ divides $|O(A)|$.  
If so, then by the Sylow Theorem, $\Gamma$ would normalize 
a Sylow $r$-subgroup $S$ of $O(A)$.   Since $z$ does not centralize
any element of order $r$, it follows that $S$ is abelian; (it is inverted by $z$).   
Let $V \subset S$ be a minimal normal $\Gamma$-invariant $r$-subgroup.   
Then $V$ must be an irreducible $\Gamma$-module.
Let $E = \mathrm{End}_{Q}(V)$.   
So $V$ is absolutely irreducible as an $EQ$-module with $E$ a field.  
Thus,  $|V|=|E|^2  < 126$ and so $|E| \le 11$.

Continuing with the assumption that a prime $r > 3$ divides $|O(A)|$,  
let $s \ge 11$ be a prime dividing $m$ (which always exists by Zygmundy's theorem and the fact that $n \ge 5$).
Let $h \in M$ be an element of order $s$; the $9$ points in ${\mathcal C}_n({\mathbb F}_4)$ are the fixed points of $h$.
Thus, $h$ must centralize $V$, whence $V$ acts trivially on the $9$ fixed points of $h$. 
Thus $V \subset M$, a contradiction.

Thus $O(A)$ is a $3$-group of order $3^a$ with $a \le 4$.
All $Q$-composition factors on $O(A)$ are nontrivial (since $z$ centralizes
no elements of $O(A)$).  It follows that $a$ is even, so $a=2$ or $4$.
Arguing as above, the element $h$ of order $s$ in $\Gamma$ must
centralize $O(A)$.  Thus, $O(A)$ acts on the $9$ fixed points of $h$.
The subgroup of $A$ fixing these $9$ points is precisely $M$, whence
$A/M$ acts faithfully on these $9$ points.  If it is not transitive,
then $A=\Gamma$.  If it is transitive, then it is $2$-transitive with
point stabilizer $\Gamma/M$, whence $A/M \cong 3^2.SL_2(3) \cong PGU_3(2)$.

Thus either $A=\Gamma$ or $[A:\Gamma]=3$.  In both cases, $M$ is normal in $A$.
\end{proof}

\begin{remark}
When $q=2$, the equation for ${\mathcal X}_n$ is the same as $y^{q+2}-y=z^m$.
In \cite[Thm.\ 3]{ABQ}, the authors show that ${\mathcal X}_n$ is a quotient of the Hermitian curve ${\mathcal H}_{q^n}$
by a cyclic group of order $2^n+1$.
We remark that \cite[Prop.\ 4.1(1)]{Duursma} is not set-up correctly, leading to a gap in the proof of \cite[Thm.\ 1.2]{Duursma};
and so it not yet clear whether ${\mathcal C}_n$ is a quotient of a Hermitian curve when $q=2$. 
\end{remark}

\subsection{The automorphism group of ${\mathcal C}_n$} \label{Smaintheorem}

Theorem \ref{Tmaintheorem} is equivalent to the following result.

\begin{theorem}
Suppose $n > 3$ is odd.
The automorphism group $\textrm{Aut}(\mathcal{C}_n)$ fixes the point at infinity on $\mathcal{C}_n$
and is isomorphic to $\Gamma=Q\rtimes_\phi \Sigma$.
\end{theorem}

\begin{proof}
Let $A=\textrm{Aut}(\mathcal{C}_n)$ and $I=\Gamma$.
When $q \not =2$, first note that the hypotheses of Notation \ref{Nsetup} are satisfied.
By Proposition \ref{TautC}, $\Gamma=Q\rtimes_\phi \Sigma \subset A$ is the inertia group of 
${\mathcal C}_n$ at $P_\infty$.  This implies that $\Gamma = N_A(Q)$.
By Proposition \ref{Csyl}, $Q$ is the Sylow $p$-subgroup of $A$ which gives condition (i).
Condition (ii) follows from the information about $Q$ in Section \ref{Q}.  
Condition (iii) is guaranteed from Proposition \ref{orbits}.

It follows that $M=Z(I)$ is normal in $A$ by Theorem \ref{Tbob} when $q \not =2$ and by Proposition \ref{Pq=2} when $q=2$.
Thus, $A/M$ embeds in $\textrm{Aut}(\mathcal{H}_q) \cong \textrm{PGU}_3(q)$.  
Since $I/M$ is a maximal subgroup of $\textrm{PGU}_3(q)$, it follows that either $A=I$ or $A/M \cong \textrm{PGU}_3(q)$.  
If the latter holds, it follows that every automorphism of $\mathcal{H}_q$ lifts to $\mathcal{C}_n$.
However, when $n >3$, this fails by Proposition \ref{Pliftw}.
Thus, $A=I$ as required.   This completes the proof.  
\end{proof}


\bibliographystyle{plain}

\bibliography{MPbib}

\vspace{.1in}
Robert Guralnick \\
Department of Mathematics, University of Southern California \\
Los Angeles, CA 90089-2532, USA

\vspace{.2in}
Beth Malmskog \\
Department of Mathematics and Computer Science, Wesleyan University \\
Middletown, CT 06459-0128, USA

\vspace{.2in}
Rachel Pries \\
Department of Mathematics, Colorado State University \\
Fort Collins, CO 80523-1874, USA 
\end{document}